\documentclass[12pt]{article}

\usepackage{amssymb}
\usepackage{amsmath}
\usepackage{url}
\usepackage{color}

\usepackage{amsthm}
\theoremstyle{definition}
\newtheorem{defn}{Definition}
\newtheorem{ex}[defn]{Example}
\newtheorem{rem}[defn]{Remark}
\theoremstyle{theorem}

\newtheorem{thm}[defn]{Theorem}

\begin{document}

\title{Switching for 2-designs}
\author{
Dean Crnkovi\'c (\url{deanc@math.uniri.hr})\\
and\\
Andrea \v{S}vob (\url{asvob@math.uniri.hr})\\
Faculty of Mathematics \\
University of Rijeka \\
Radmile Matej\v ci\'c 2, 51000 Rijeka, Croatia
}

\maketitle

\begin{abstract}
In this paper, we introduce a switching for 2-designs, which defines a type of trade. We illustrate this method by applying it to some symmetric $(64,28,12)$ designs, showing that the switching introduced in this paper in some cases can be applied directly to orbit matrices. In that way we obtain six new symmetric $(64,28,12)$ designs. 
Further, we show that this type of switching (of trades) can be applied to any symmetric design related to a Bush-type Hadamard matrix and construct symmetric designs with parameters $(36,15,6)$ leading to new Bush-type Hadamard matrices of order 36, and symmetric $(100,45,20)$ designs yielding Bush-type Hadamard matrices of order 100.

\end{abstract}
\vspace*{0.2cm}

\vspace*{0.2cm}


\section{Introduction and preliminaries}

Switching methods have been successfully used for constructing strongly regular graphs (see \cite{abh, blo, G-McK, Seidel}). Denniston \cite{denniston} used a method called switching ovals for a
construction of symmetric $(25,9,3)$ designs. A similar idea was employed in \cite{orrick}, where Orrick defined switching operations for Hadamard matrices. The switching using Pasch configurations, so called Pasch switch, was used for a construction of new Steiner triple systems from known ones (see \cite{gibbons, g-g}).
Norton \cite{norton}, Parker \cite{parker} and Wanless \cite{wanless} used switching for a construction of Latin squares.
Further, \"Osterg\r ard in \cite{prjo} introduced a switching for codes and Steiner systems.
In this paper, we introduce a switching that can be applied to any 2-design having a set of blocks satisfying certain properties, and give few examples.
In \cite{dj-vdt}, Jungnickel and Tonchev used maximal arcs for a transformation of quasi-symmetric designs that leads to a construction of new quasi-symmetric designs.
This transformation can also be described as switching, but it is different than the switching introduced in this paper.

\bigskip

A $2$-$(v,k, \lambda)$ design is a finite incidence structure ${\mathcal{D}}=(\mathcal{P}, \mathcal{B}, \mathcal{I})$, where $\mathcal{P}$ and $\mathcal{B}$ are disjoint sets and 
$\mathcal{I} \subseteq \mathcal{P} \times \mathcal{B}$, with the following properties:
\begin{enumerate}
  \item $|\mathcal{P}|=v$ and $1< k <v-1$,
  \item every element (block) of $\mathcal{B}$ is incident with exactly $k$ elements (points) of $\mathcal{P}$,
	\item every two distinct points in $\mathcal{P}$ are together incident with exactly $\lambda$ blocks of $\mathcal{B}$.
\end{enumerate}

In a $2$-$(v,k,\lambda )$ design, every point is incident with exactly $\displaystyle r=\frac{\lambda (v-1)}{k-1}$ blocks, and $r$ is called the replication number of a design.
A 2-design is also called a block design. The number of blocks in a block design is denoted by $b$. If $v=b$, a design is called symmetric. 
An isomorphism from one block design to another is a bijective mapping of points to points and blocks to blocks which preserves incidence.
An isomorphism from a block design ${\mathcal{D}}$ onto itself is called an automorphism of ${\mathcal{D}}$. The set of all automorphisms forms a group called the full automorphism group
of ${\mathcal{D}}$ and is denoted by $Aut({\mathcal{D}})$. 
If $\mathcal{D}$ is a block design, the incidence structure $\mathcal{D}'$ having as points the blocks of $\mathcal{D}$, and having as blocks the points of $\mathcal{D}$, where a point and a block are incident in $\mathcal{D}'$ if and only if the corresponding block and a point of $\mathcal{D}$ are incident, is a block design called the dual of $\mathcal{D}$. 
The dual design of a 2-design $\mathcal{D}$ is a 2-design if and only if $\mathcal{D}$ is symmetric. A symmetric design is called self-dual if the design and its dual are isomorphic. 
Let $A$ be the $b\times v$ (block-point) incidence matrix of a 2-design. The rank of $A$ over $GF(p)$ is called the $p$-rank of $A$. For further reading on block designs we refer the reader to 
\cite{bjl}.

The switching introduced in this paper defines a trade. 
A trade for a 2-$(v,k, \lambda)$ design consists of two disjoint sets of blocks with the property that if the design contains the blocks of one of the sets, 
then these blocks can be replaced by the blocks of the other set (see \cite{prjo}). 
For example, if a design has a subdesign, then the subdesign can be replaced by a disjoint subdesign with the same parameters (see \cite{prjo}). 
For more results on trades we refer the reader to \cite{ejb, trades}. In Section \ref{men-64} we show that switching (of trades) introduced in this paper can be applied directly to orbit matrices.

The paper is outlined as follows. In Section \ref{switch}, we introduce a switching for 2-designs. In Section \ref{appl}, we apply the switching introduced to construct new designs from known symmetric 
$(64,28,12)$ designs. Further, we show that this switching can be applied to symmetric designs related to Bush-type Hadamard matrices and illustrate that, by constructing examples of Bush-type Hadamard matrices of orders 36 and 100. Section \ref{con} contains concluding remarks.   

For the computations we have used computer algebra system MAGMA \cite{magma}.

\section{Switching for 2-designs} \label{switch}

In this section, we define a switching for 2-designs and show that an incidence structure obtained from a 2-$(v,k, \lambda)$ design by switching is also a 2-$(v,k, \lambda)$ design.

\begin{defn} \label{defn-switch}
Let ${\mathcal{D}}=(\mathcal{P}, \mathcal{B}, \mathcal{I})$ be a 2-design and let $\mathcal{B}_1 \subset \mathcal{B}$ be a set of blocks such that there are sets of points 
$\mathcal{P}_1, \mathcal{P}_2 \subset \mathcal{P}$ with the following properties:
\begin{enumerate}
 \item $(P,B) \notin \mathcal{I}$, for every $(P,B) \in \mathcal{P}_1 \times \mathcal{B}_1$,
 \item $(P,B) \in \mathcal{I}$, for every $(P,B) \in \mathcal{P}_2 \times \mathcal{B}_1$,
 \item $|\{ B \in \mathcal{B}_1: (P,B) \in \mathcal{I} \}|=|\{ B \in \mathcal{B}_1: (P,B) \notin \mathcal{I} \}|$, for every $P \in \mathcal{P} \backslash (\mathcal{P}_1 \cup \mathcal{P}_2)$.
\end{enumerate}
Then $\mathcal{B}_1$ is called a switching set of ${\mathcal{D}}$. 
\end{defn}

Note that for a switching set $\mathcal{B}_1$ the sets $\mathcal{P}_1$ and $\mathcal{P}_2$ with the above properties are uniquely determined. If $\mathcal{B}_1$ is a switching set of a 2-$(v,k, \lambda)$ design ${\mathcal{D}}=(\mathcal{P}, \mathcal{B}, \mathcal{I})$, we define an incidence structure  ${\mathcal{D}_1}=(\mathcal{P}, \mathcal{B}, \mathcal{I}_1)$ obtained from ${\mathcal{D}}$ by
switching with respect to $\mathcal{B}_1$ in the following way:
\begin{enumerate}
 \item $(P,B) \in \mathcal{I}_1 \Leftrightarrow (P,B) \in \mathcal{I}$, for $B \in \mathcal{B} \backslash \mathcal{B}_1$, $P \in \mathcal{P}$,
 \item $(P,B) \in \mathcal{I}_1 \Leftrightarrow (P,B) \in \mathcal{I}$, for $B \in \mathcal{B}_1$, $P \in \mathcal{P}_1 \cup \mathcal{P}_2$,
 \item $(P,B) \in \mathcal{I}_1 \Leftrightarrow (P,B) \notin \mathcal{I}$, for $B \in \mathcal{B}_1$, $P \in \mathcal{P} \backslash (\mathcal{P}_1 \cup \mathcal{P}_2)$.
\end{enumerate}

\begin{thm} \label{thm-switch}
Let ${\mathcal{D}}=(\mathcal{P}, \mathcal{B}, \mathcal{I})$ be a $2$-$(v,k, \lambda)$ design. If $\mathcal{B}_1$ is a switching set of ${\mathcal{D}}$ then the incidence structure 
${\mathcal{D}_1}=(\mathcal{P}, \mathcal{B}, \mathcal{I}_1)$ obtained from ${\mathcal{D}}$ by switching with respect to $\mathcal{B}_1$ is also a $2$-$(v,k, \lambda)$ design.
\end{thm}
\begin{proof}
Let us show that all blocks of ${\mathcal{D}_1}$ are incident with precisely $k$ points. It is clear that a block $B \in \mathcal{B} \backslash \mathcal{B}_1$ of ${\mathcal{D}_1}$ is incident with $k$
points. Let $B$ belong to the switching set $\mathcal{B}_1$. The block $B$ is incident in ${\mathcal{D}}$ with $k-|\mathcal{P}_2|$ points from $\mathcal{P} \backslash (\mathcal{P}_1 \cup \mathcal{P}_2)$. 
Further, it follows directly from the definition of a switching set that
$$|\{(P,B) \in \mathcal{I}: P \in \mathcal{P} \backslash (\mathcal{P}_1 \cup \mathcal{P}_2), B \in \mathcal{B}_1 \}|= 
\frac{|\mathcal{P} \backslash (\mathcal{P}_1 \cup \mathcal{P}_2) \times \mathcal{B}_1|}{2}.$$
Therefore, 
$$k-|\mathcal{P}_2|=\frac{|\mathcal{P} \backslash (\mathcal{P}_1 \cup \mathcal{P}_2)|}{2}.$$
It follows that the block $B$ is incident in ${\mathcal{D}_1}$ with $k-|\mathcal{P}_2|$ points from the set $\mathcal{P} \backslash (\mathcal{P}_1 \cup \mathcal{P}_2)$. 
Hence, any block from ${\mathcal{B}_1}$ is incident in ${\mathcal{D}_1}$ with $k$ points.

It remains to prove that every two distinct points are together incident in ${\mathcal{D}_1}$ with $\lambda$ blocks. It is obvious that two points from $\mathcal{P}_1 \cup \mathcal{P}_2$ 
are together incident with exactly $\lambda$ blocks. Further, two points, one from $\mathcal{P}_1 \cup \mathcal{P}_2$ and one from $\mathcal{P} \backslash (\mathcal{P}_1 \cup \mathcal{P}_2)$,
are also together incident with $\lambda$ blocks. Let $P_1$ and $P_2$ be points from the set $\mathcal{P} \backslash (\mathcal{P}_1 \cup \mathcal{P}_2)$. Denote by $x$ the number of blocks
from $\mathcal{B}_1$ that are incident in ${\mathcal{D}}$ with $P_1$ and $P_2$. Then there are $x$ blocks of $\mathcal{B}_1$ that are not incident in ${\mathcal{D}}$ neither with $P_1$ nor with $P_2$.
Hence, there are $x$ blocks of $\mathcal{B}_1$ that are incident in ${\mathcal{D}_1}$ with $P_1$ and $P_2$. Since there are $\lambda-x$ blocks from $\mathcal{B} \backslash \mathcal{B}_1$ that 
are incident with $P_1$ and $P_2$ in both designs ${\mathcal{D}}$ and ${\mathcal{D}_1}$, it follows that $P_1$ and $P_2$ are together incident with $\lambda$ blocks in ${\mathcal{D}_1}$.
So, ${\mathcal{D}_1}$ is a $2$-$(v,k, \lambda)$ design.
\end{proof}

Obviously, if a design ${\mathcal{D}}_1$ is obtained from ${\mathcal{D}}$ by switching with respect to $\mathcal{B}_1$, 
then ${\mathcal{D}}$ can be obtained form ${\mathcal{D}_1}$ also by switching with respect to $\mathcal{B}_1$.
If 2-designs ${\mathcal{D}}$ and ${\mathcal{D}}_1$ can be obtained from each other by switching, then ${\mathcal{D}}$ and ${\mathcal{D}}_1$ are said to be switching-equivalent.

\begin{rem} \label{rem-trades}
If ${\mathcal{B}}_1$ is a switching set of a symmetric 2-design ${\mathcal{D}}$, then the incident structure with the point set ${\mathcal{B}}_1$
and the block set $\mathcal{P} \backslash (\mathcal{P}_1 \cup \mathcal{P}_2)$ is a 2-design which is a subdesign of the dual design of ${\mathcal{D}}$.
\end{rem}

\begin{rem}
Clearly, the size of a switching set is an even number. Note that any pair of blocks of a 2-design form a switching set. However, if ${\mathcal{D}}$ and ${\mathcal{D}_1}$ are switching-equivalent with respect to a switching set of size two, then they are isomorphic. Switching with respect to switching sets of size greater than two usually produce non-isomorphic designs.
\end{rem}

\section{Application to some symmetric designs} \label{appl}

In this section, we apply the method introduced to symmetric designs obtained from Bush-type Hadamard matrices, and to certain symmetric designs with parameters $(64,28,12)$.
It is clear that the switching does not have to preserve the full automorphism group of a design. The examples given in this section show that the switching introduced in this paper does not preserve 
$p$-rank of a 2-design, in case when $p$ divides the order of the design, and also does not preserve the self-duality of a symmetric design.

\subsection{Examples from symmetric $(64,28,12)$ designs} \label{men-64}

A Hadamard matrix of order $m$ is an $(m \times m)$-matrix $H=( h_{i,j})$, $h_{i,j} \in \{ -1,1 \}$, satisfying $HH^T=H^TH=mI_m$, where $I_m$ is the unit matrix of order $m$.  
Two Hadamard matrices are said to be equivalent if one can be obtained from the other by negating rows or columns, or by interchanging rows or columns.
A Hadamard matrix is regular if the row and column sums are constant.
It is well known that the existence of a symmetric design with parameters $(4u^2, 2u^2-u, u^2-u)$ is equivalent to the existence of a regular Hadamard matrix of order $4u^2$ 
(see \cite[Theorem 1.4 p. 280]{w}). Symmetric designs with parameters $(4u^2, 2u^2-u, u^2-u)$ are called Menon designs.

Symmetric $(64,28,12)$ designs are related to regular Hadamard matrices of order 64. To illustrate the method introduced in this paper, we will apply the switching to some of the symmetric $(64,28,12)$ designs 
described in \cite{dc-mop}. The examples where done by computer, using the computer algebra system MAGMA \cite{magma}.

\begin{ex}
It is shown in \cite{dc-mop} that up to isomorphism there are exactly two orbit matrices for the action of the Frobenius group of order $21$, i.e., the nonabelian group of order 21, 
on a symmetric $(64,28,12)$ design so that its cyclic subgroup of order $3$ fixes exactly $10$ points. 
These block-by-point orbit matrices $M_1$ and $M_2$ are given below, where the first row and the first column gives the orbit lengths. The designs obtained from $M_1$ 
are not isomorphic to the designs obtained from $M_2$, but we will show that they are related by switching. 

{\tiny
\begin{center}
\begin{tabular}{c|cccccccccc}
$M_1$ & 1 & 7 & 7 & 7 & 7 & 7 & 7 & 7 & 7 & 7 \\ \hline
  1 & 0 & 7 & 7 & 7 & 7 & 0 & 0 & 0 & 0 & 0 \\
  7 & 1 & 4 & 4 & 4 & 0 & 3 & 3 & 3 & 3 & 3 \\
  7 & 1 & 4 & 4 & 0 & 4 & 3 & 3 & 3 & 3 & 3 \\
  7 & 1 & 4 & 0 & 4 & 4 & 3 & 3 & 3 & 3 & 3 \\
  7 & 1 & 0 & 4 & 4 & 4 & 3 & 3 & 3 & 3 & 3 \\
  7 & 0 & 3 & 3 & 3 & 3 & 4 & 4 & 4 & 4 & 0 \\
  7 & 0 & 3 & 3 & 3 & 3 & 4 & 4 & 4 & 0 & 4 \\
  7 & 0 & 3 & 3 & 3 & 3 & 4 & 4 & 0 & 4 & 4 \\
  7 & 0 & 3 & 3 & 3 & 3 & 4 & 0 & 4 & 4 & 4 \\
  7 & 0 & 3 & 3 & 3 & 3 & 0 & 4 & 4 & 4 & 4 \\
\end{tabular} \qquad
\begin{tabular}{c|cccccccccc}
$M_2$ & 1 & 7 & 7 & 7 & 7 & 7 & 7 & 7 & 7 & 7 \\ \hline
  1 & 0 & 7 & 7 & 7 & 7 & 0 & 0 & 0 & 0 & 0 \\
  7 & 1 & 3 & 3 & 3 & 3 & 4 & 4 & 4 & 3 & 0 \\
  7 & 1 & 3 & 3 & 3 & 3 & 4 & 4 & 0 & 3 & 4 \\
  7 & 1 & 3 & 3 & 3 & 3 & 4 & 0 & 4 & 3 & 4 \\
  7 & 1 & 3 & 3 & 3 & 3 & 0 & 4 & 4 & 3 & 4 \\
  7 & 0 & 4 & 4 & 4 & 0 & 3 & 3 & 3 & 4 & 3 \\
  7 & 0 & 4 & 4 & 0 & 4 & 3 & 3 & 3 & 4 & 3 \\
  7 & 0 & 4 & 0 & 4 & 4 & 3 & 3 & 3 & 4 & 3 \\
  7 & 0 & 0 & 4 & 4 & 4 & 3 & 3 & 3 & 4 & 3 \\
	7 & 0 & 3 & 3 & 3 & 3 & 4 & 4 & 4 & 0 & 4 \\
\end{tabular}
\end{center}
}
As shown in \cite{dc-mop}, up to isomorphism there are $10$ mutually non-isomorphic symmetric designs with parameters $(64,28,12)$ admitting the action of the Frobenius group of order $21$ so that the cyclic group of order $3$ fixes $10$ points and blocks. The orbit matrix $M_1$ leads to 3 symmetric designs, denoted by ${\cal D}_1$, ${\cal D}_2$ and ${\cal D}_3$, 
and the orbit matrix $M_2$ corresponds to 7 symmetric designs, denoted by  ${\cal D}_4, \dots ,{\cal D}_{10}$. 
It can be seen from the orbit matrix $M_1$ that the fixed blocks of ${\cal D}_1$, ${\cal D}_2$ and ${\cal D}_3$, together with any of the orbits of length 7
corresponding to the last 5 rows of $M_1$, form switching sets of size 8. Further, the fixed blocks of the designs ${\cal D}_4, \dots ,{\cal D}_{10}$, together with the orbit of length 7
corresponding to the last row of $M_2$, also form switching sets of size 8. Switching of the designs ${\cal D}_1$, ${\cal D}_2$ and ${\cal D}_3$ with respect to the described switching sets
produce designs ${\cal D}_4, \dots ,{\cal D}_{10}$. The designs ${\cal D}_4, \dots ,{\cal D}_{10}$ have 2-rank 26, and the designs ${\cal D}_4, \dots ,{\cal D}_{10}$ have 2-rank 27. That means that
the switching does not preserve $p$-rank of a design. Further, the designs ${\cal D}_1$, ${\cal D}_2$, ${\cal D}_3$, ${\cal D}_4$, ${\cal D}_9$ and ${\cal D}_{10}$ are self-dual, while 
$({\cal D}_5,{\cal D}_7)$ and $({\cal D}_6,{\cal D}_8)$ are pairs of mutually dual designs. It means that the property of being self-dual is not preserved under the switching introduced in this paper. 
\end{ex}

\begin{ex}
Up to isomorphism, there are nine orbit matrices for the Frobenius group of order $21$ acting on a symmetric $(64,28,12)$ design in a way that its cyclic subgroup of order $3$ fixes $7$ points 
(see \cite{dc-mop}). One of these nine orbit matrices is the orbit matrix $M_3$ given below.

{\tiny
\begin{center}
\begin{tabular}{c|c c c c c c c c}
$M_3$ & 1 & 7 & 7 & 7 & 7 & 7 & 7 & 21\\ \hline
1  & 0 & 7 & 7 & 7 & 7 & 0 & 0 &  0 \\
7  & 1 & 4 & 4 & 4 & 0 & 3 & 3 &  9 \\
7  & 1 & 4 & 4 & 0 & 4 & 3 & 3 &  9 \\
7  & 1 & 4 & 0 & 4 & 4 & 3 & 3 &  9 \\
7  & 1 & 0 & 4 & 4 & 4 & 3 & 3 &  9 \\
7  & 0 & 3 & 3 & 3 & 3 & 4 & 0 & 12 \\
7  & 0 & 3 & 3 & 3 & 3 & 0 & 4 & 12 \\
21 & 0 & 3 & 3 & 3 & 3 & 4 & 4 &  8 \\
\end{tabular} 
\qquad
\begin{tabular}{c|c c c c c c c c c c}
$M_3'$ & 1 & 7 & 7 & 7 & 7 & 7 & 7 & 7 & 7 & 7 \\ \hline
1  & 0 & 7 & 7 & 7 & 7 & 0 & 0 & 0 & 0 & 0 \\
7  & 1 & 4 & 4 & 4 & 0 & 3 & 3 & 3 & 3 & 3 \\
7  & 1 & 4 & 4 & 0 & 4 & 3 & 3 & 3 & 3 & 3 \\
7  & 1 & 4 & 0 & 4 & 4 & 3 & 3 & 3 & 3 & 3 \\
7  & 1 & 0 & 4 & 4 & 4 & 3 & 3 & 3 & 3 & 3 \\ 
7  & 0 & 3 & 3 & 3 & 3 & 4 & 0 & 4 & 4 & 4 \\
7  & 0 & 3 & 3 & 3 & 3 & 0 & 4 & 4 & 4 & 4 \\
7  & 0 & 3 & 3 & 3 & 3 & 4 & 4 & 4 & 4 & 0 \\
7  & 0 & 3 & 3 & 3 & 3 & 4 & 4 & 0 & 4 & 4 \\
7  & 0 & 3 & 3 & 3 & 3 & 4 & 4 & 4 & 0 & 4 \\
\end{tabular}
\end{center}
}

If the Frobenius group of order $21$ acts on a symmetric $(64,28,12)$ design yielding the orbit matrix $M_3$, then its subgroup of order 7 acts on that designs producing the orbit matrix $M_3'$.
The orbit matrix $M_3$ leads to three symmetric $(64,28,12)$ designs, denoted in \cite{dc-mop} by ${\cal D}_{27}$, ${\cal D}_{28}$ and ${\cal D}_{29}$.
Any of the last five orbits for $Z_7$ of the designs ${\cal D}_{27}$, ${\cal D}_{28}$ and ${\cal D}_{29}$, together with the fixed block, form a switching set of size 8.
Switching of ${\cal D}_{27}$, ${\cal D}_{28}$ and ${\cal D}_{29}$ with respect to the switching sets obtained from last three orbits of length 7 (that belong to the orbit of length 21 of the Frobenius group of order 21) produce symmetric $(64,28,12)$ designs that are not isomorphic to the designs obtained in \cite{dc-mop}. From each of the designs ${\cal D}_{27}$, ${\cal D}_{28}$ and ${\cal D}_{29}$
we obtain, up to isomorphism, one new design denoted by ${\cal D}_{27}'$, ${\cal D}_{28}'$ and ${\cal D}_{29}'$, respectively. The designs ${\cal D}_{27}'$, ${\cal D}_{28}'$ and ${\cal D}_{29}'$
are pairwise non-isomorphic. While the designs ${\cal D}_{27}$, ${\cal D}_{28}$ and ${\cal D}_{29}$ are self-dual, the newly obtained designs ${\cal D}_{27}'$, ${\cal D}_{28}'$ and ${\cal D}_{29}'$
are not self-dual, and together with their duals give us six designs that are not isomorphic to the designs obtained in (see \cite{dc-mop}). 
The full automorphism group of ${\cal D}_{27}'$ is isomorphic to $Z_7 \times Z_2$, and the full automorphism groups of ${\cal D}_{28}'$ and ${\cal D}_{29}'$ are isomorphic to $Z_7$.
While the designs ${\cal D}_{27}$, ${\cal D}_{28}$ and ${\cal D}_{29}$ have 2-rank equal to 26, the 2-rank of any of the design ${\cal D}_{28}'$ and ${\cal D}_{29}'$ is 27.
\end{ex}

\begin{rem}
The designs ${\cal D}_{27}'$, ${\cal D}_{28}'$ and ${\cal D}_{29}'$ are not isomorphic to the designs described in \cite{dilon, kraemer, p-s-t}, since all these designs have the full automorphism groups 
of order divisible by $2^6$. Further, none of the derived designs (with parameters $2$-$(28,12,11)$) of ${\cal D}_{27}'$, ${\cal D}_{28}'$ and ${\cal D}_{29}'$ and their duals is quasi-symmetric, and therefore
the designs ${\cal D}_{27}'$, ${\cal D}_{28}'$ and ${\cal D}_{29}'$ and their duals are not isomorphic to the designs obtained in \cite{ding&al, l-t-t}.
The designs ${\cal D}_{27}'$, ${\cal D}_{28}'$ and ${\cal D}_{29}'$ and their duals are also not isomorphic to the symmetric $(64,28,12)$ design from \cite{dc}.
Therefore, the designs ${\cal D}_{27}'$, ${\cal D}_{28}'$ and ${\cal D}_{29}'$ and their duals are possibly not isomorphic to the previously known symmetric $(64,28,12)$ designs.
\end{rem}

\subsection{Symmetric designs related to Bush-type Hadamard matrices} \label{bush-type}

We will show that the switching introduced in this paper can be applied to any symmetric design obtained from a Bush-type Hadamard matrix.

A Bush-type Hadamard matrix of order $4n^2$ is a Hadamard matrix with the additional property of being a block matrix  $H=[H_{i,j}]$ with blocks of size $2n \times 2n$, such that
$H_{i,i}=J_{2n}$ and $H_{i,j}J_{2n}=J_{2n}H_{i,j}=0,\ i \neq j,$ $1 \le i \le 2n,\ 1 \le j \le 2n$, where $J_{2n}$ is the all-ones $(2n \times 2n)$-matrix.

Bush-type Hadamard matrices of order $16n^2$ exist for all values of $n$ for which a Hadamard matrix of order $4n$ exists (see \cite{kh}). However, it
is very difficult to decide whether such matrices of order $4n^2$ exist if $n$ is an odd integer, $n>1$. Bush-type Hadamard matrices of order $4n^2$, where $n$ is an odd integer, 
have been constructed for $n=3,5,9$ (see \cite{janko-36, jkt, b-t-324}). Later on, M. Muzychuk and Q. Xiang gave a construction of Bush-type Hadamard matrices of order $4n^4$ for any odd $n$ \cite{m-x}.
While there are at least 52432 symmetric $(100,45,20)$ designs corresponding to Bush-type Hadamard matrices of order 100 (see \cite{dc-dh}),
the only known symmetric designs with parameters $(36,15,6)$ and $(324,153,72)$ corresponding to Bush-type Hadamard matrices of order 36 and 324, respectively, are two designs corresponding to 
Bush-type Hadamard matrices of order 36 given in \cite{janko-36, janko-hadi}, and one design corresponding to Bush-type Hadamard matrices of order 324 given in \cite{b-t-324}.

H. Kharaghani \cite{kh} showed that a Bush-type Hadamard matrix of order $4n^2$ with $2n-1$ or $2n+1$ a prime power, can be used to construct infinite classes of symmetric designs.
Further, Janko and Kharaghani constructed strongly regular graphs with parameters $(936,375,150,150)$ and $(1800,1029,588,588)$ from a block negacyclic Bush-type Hadamard matrix of order 36 
(see \cite{janko-hadi}).

Bush-type Hadamard matrices are regular. By replacing the entries 1 of a Bush-type Hadamard matrix of order $4n^2$ with $0$, and the entries equal to $-1$ with $1$, one gets the incidence
matrix of a symmetric design with parameters $(4n^2, 2n^2-n, n^2-n)$. The diagonal blocks of these incidence matrices are zero matrices, and the off-diagonal blocks have the same number of 1s and 0s in
each row and column. Hence, each $2n \times 2n$ diagonal block of the Menon design corresponding to a Bush-type Hadamard matrix of order $4n^2$ determines a switching set, which give us $2n$ switching
sets. Each of these $2n$ switching sets can be used for switching, producing $2^n-1$ designs that are switching-equivalent with the starting design. Note that these $2^n-1$ designs also correspond to
Bush-type Hadamard matrices.

Bush-type Hadamard matrices of order $4n^2$, where $n$ is an odd prime, have been constructed only for $n=3,5$.  In Examples
\ref{Bush36-negacyclic}, \ref{Bush36-Z3} and \ref{Bush100}, we construct new symmetric designs corresponding to Bush-type Hadamard matrices of order 36 and 100.

\begin{ex} \label{Bush36-negacyclic}
To illustrate the construction method, we use the block negacyclic Bush-type Hadamard matrix of order 36 from \cite{janko-hadi}. 
The six diagonal blocks of that Bush-type Hadamard matrix determine six switching sets of the corresponding Menon design, and switching leads us to 64 symmetric $(36,15,6)$ designs 
(including the starting one). These 64 designs are pairwise non-isomorphic, and all of them have trivial full automorphism group. 
Any of these 64 symmetric $(36,15,6)$ designs can be used to construct two infinite classes of symmetric designs, as shown in \cite{janko-36}. Note that the switching with respect to all six
switching sets produces a design which corresponds to a Bush-type Hadamard matrix of order 36 that is also block negacyclic. 

For a prime $p$, $p$-rank of the incidence matrix of a 2-$(v, k, \lambda)$ design can be smaller than $v-1$ only if $p$ divides the order of a design, i.e. if $p$ divides $k - \lambda$ in
the case of a symmetric design (see \cite{hamada}). The symmetric $(36,15,6)$ design constructed in \cite{janko-hadi} has 3-rank equal to 15, while 10 of the designs obtained by switching have 
3-rank equal to 16, 28 of them have 3-rank 17, and 18 designs have 3-rank equal to 18. 

The 64 Bush-type Hadamard matrices corresponding to the 64 Menon designs obtained in this example form 14 equivalence classes.
\end{ex}

\begin{ex} \label{Bush36-Z3}
In this example, we use the Bush-type Hadamard matrix of order 36 from \cite{janko-36}. The six diagonal blocks of that Bush-type Hadamard matrix determine six switching sets of the corresponding design, and switching leads us to 64 symmetric $(36,15,6)$ designs 
(including the starting one). In total, 24 designs of these 64 designs are pairwise non-isomorphic, 20 of them have trivial full automorphism group while 4 of them have the automorphism group of order 3.
Any of these 24 symmetric $(36,15,6)$ designs can be used to construct two infinite classes of symmetric designs, as shown in \cite{janko-36}. 
The symmetric $(36,15,6)$ design constructed in \cite{janko-36} has 3-rank equal to 16, one design obtained by switching has 
3-rank equal to 15, 4 of them have 3-rank 16, 10 of them have 3-rank equal to 17 and 8 designs have 3-rank equal to 18. 

The 24 Bush-type Hadamard matrices corresponding to the 24 symmetric $(36,15,6)$ designs obtained in this example form 16 equivalence classes.
\end{ex}

\begin{ex} \label{Bush100}
In this example, we use the Bush-type Hadamard matrix of order 100 from \cite{jkt}. The ten diagonal blocks of that Bush-type Hadamard matrix determine ten switching sets of the corresponding design, and switching leads us to 1024 symmetric $(100,45,20)$ designs 
(including the starting one). In total, 208 of these 1024 designs are pairwise non-isomorphic, 204 of them have the full automorphism group of size 20, while 4 of them have the automorphism group of order 100. Two designs have 5-rank equal to 38, 4 of them have 5-rank equal to 39, 20 of the designs have 5-rank equal to 40, 64 of them have 5-rank 41 and 118 designs have 5-rank equal to 42. 

The 208 Bush-type Hadamard matrices corresponding to the 208 symmetric $(100,45,20)$ designs obtained in this example, including the design constructed in \cite{jkt}, form 120 equivalence classes.
\end{ex}

\section{Conclusion} \label{con}

In this paper, we introduce a switching for 2-designs, and show that this switching can be used to produce non-isomorphic designs. In that way, we obtain six symmetric $(64,28,12)$ designs. Further, we construct 86 pairwise non-isomorphic symmetric $(36,15,6)$ designs leading to 28 new pairwise nonequivalent Bush-type Hadamard matrices of order 36, and 207 pairwise non-isomorphic symmetric $(100,45,20)$ designs leading to 119 pairwise nonequivalent Bush-type Hadamard matrices of order 100.
Moreover, the switching applied to Menon designs in Section \ref{appl} show that the switching
may lead to construction of inequivalent Hadamard matrices. The examples given in this paper show that the switching does not preserve 
$p$-rank of a 2-design, in case when $p$ divides the order of the design, and also does not preserve the self-duality of a symmetric design.

\section{Statements and Declarations}

\subsection{Funding}
This work has been fully supported by {\rm C}roatian Science Foundation under the projects 6732 and 5713.

\end{document}